\documentclass[12pt]{article}

\usepackage{algorithm,algpseudocode}
\usepackage{amscd,amsfonts,amsopn,amssymb,amstext}
\usepackage{appendix}
\usepackage{booktabs}
\usepackage{color}
\usepackage{fullpage}
\usepackage{graphicx,graphics,psfrag}
\usepackage{latexsym,enumerate}
\usepackage{multirow}
\usepackage[square]{natbib}
\setcitestyle{numbers,square}
\usepackage{setspace}
\usepackage[T1]{fontenc}
\usepackage{times}
\usepackage{url}
\usepackage{bm}

\makeatletter
\renewcommand\@biblabel[1]{#1.}
\makeatother
\newtheorem{theorem}{Theorem}
\newenvironment{proof}{{\bf Proof:\ \ }}{\qed}
\newcommand{\qed}{\rule{0.5em}{1.5ex}}
\title{\Large \bf{Multiple integral representations of the Catalan's constant}}

\author{\normalsize
{\bf Emilio G\'omez-D\'eniz$\,^a$, Jos\'e Mar\'ia Sarabia$\,^b$}\\
{\small $\,^a$Department of Quantitative Methods in Economics and TIDES Institute}\\[-0.2cm]
{\small University of Las Palmas de Gran Canaria, Spain.}\\[-0.2cm]
{\small https://orcid.org/0000-0002-5072-7908. \url{emilio.gomez-deniz@ulpgc.es}.}  \\[-0.2cm]
{\small $\,^b$Department of Economics, University of Cantabria, Santander, Spain.}\\[-0.2cm]
{\small https://orcid.org/0000-0002-9619-4721. \url{jose.sarabia@unican.es}}\\[-0.2cm]
}

\date{}

\def \E{{\rm I\kern -2.2pt  E}}

\begin{document}

\maketitle

\begin{abstract}
\noindent In this paper, we present several novel integral representations of Catalan's constant. We begin by deriving an initial result expressed as a double integral. Subsequently, as a consequence of this result, we establish a general theorem that enables the representation of Catalan's constant in terms of a single integral. Finally, we provide a multiple integral representation of Catalan's constant in dimensions greater than or equal to two using the Lerch function. The results are accompanied by illustrative examples.

%A simple multiple integral representation of the Catalan's constant is provided. Several examples illustrate how to use the new result.

\medskip


\medskip

\paragraph{MR (2010):}  11A99, 11M35, 11Z99

\end{abstract}

\medskip

\paragraph{Funding:} EGD was partially funded by grant PID2021-127989OB-I00 (Ministerio de
Econom\'ia y Competitividad, Spain). JMS acknowledges financial Support from the I+D+i project Ref. PID2024-156871NB-I00 financed by MICIU/AEI/10.13039/501100011033/FEDER,UE.

\medskip

\paragraph{Corresponding author:} Emilio G\'omez-D\'eniz. Department of Quantitative Methods in Economics and TIDES Institute. Campus Universitario de Tafira. Edificio de Ciencias Econ\'omicas y Empresariales, 35017.  University of Las Palmas de Gran Canaria, Las Palmas, Canary Islands, Spain. E-mail: {\tt emilio.gomez-deniz@ulpgc.es}

\newpage

\section{Introduction}

Catalan's constant, denoted as ${\bf G}$, was introduced by the mathematician Eug\`ene Charles Catalan and is defined by the alternating series,
\begin{eqnarray*}\label{f1}
{\bf G}=\sum_{n=0}^{\infty}\frac{(-1)^n}{(2n+1)^2},
\end{eqnarray*}
with an approximate value of ${\bf G}\approx 0.915965.$ The constant can also be expressed in terms of the definite integral,
\begin{eqnarray*}\label{f2}
{\bf G}=\int_{0}^{1}\frac{\tan^{-1}x}{x}dx,
\end{eqnarray*}
which was widely used by Ramanujan. The variety of integrals and series from which {\bf G} can be obtained is surprising. Two compilations that include an extensive list of formulas involving Catalan's constant can be found in \cite{Bradley2001} and \cite{Marichevetal2025}, this last reference from the Wolfram Mathworld website. Some additional references that include integrals and series leading to the Catalan's constant are \cite{Chu2023}, \cite{ferretietal_2020}, \cite{JamesonLord2017}, \cite{Mortini2024}, and \cite{Stewart2020} and more recently \cite{Safonova2025}.

The structure of this article is as follows. Section \ref{section2} introduces a representation of the Catalan constant expressed through double integrals, along with several illustrative examples. In Section \ref{section3}, a novel representation of the Catalan constant as a single integral is presented. Finally, Section \ref{section4} provides a general formulation of the Catalan constant in terms of multiple integrals, accompanied by representative examples.

\section{Representation in terms of double integrals}\label{section2}

The aim of this paper is based on the next theorem.

\begin{theorem}\label{th1}
Let $G_i:\mathbb{R}\to [0,1]$, $i=1,2$ be two right-continuous monotone non-decreasing functions satisfying $\lim_{x\to -\infty}G_i(x)=0$ and $\lim_{x\to\infty}G_i(x)=1$. If $G_i(x)=1-G_i(-x)$, $i=1,2$ and $a>0$ we have,
\begin{equation}\label{eq1}
{\bf G}=\int_{-a}^{a}\int_{-1/a}^{1/a}\frac{G_1(x_1)G_2(x_2)}{1+x_1^2x_2^2}\,dx_1\,dx_2,
\end{equation}
where ${\bf G}$ is the Catalan's constant.
\end{theorem}
\begin{proof} Let consider $\int_{-a}^{a}x^{2n}G(x)\,dx$ with $a>0$. Applying integration by parts, we get
\begin{eqnarray*}
\int_{-a}^{a}x^{2n}G(x)\,dx &=& \left.\frac{x^{2n+1}}{2n+1}G(x)\right|_{-a}^{a}-\frac{1}{2n+2}\int_{-a}^{a}x^{2n+1}G^{\prime}(x)\,dx\\
&=& \frac{a^{2n+1}}{2n+1}G(a)-\frac{(-a)^{2n+1}}{2n+1}G(-a)=\frac{a^{2n+1}}{2n+1}[G(a)+G(-a)]=\frac{a^{2n+1}}{2n+1},
\end{eqnarray*}
where we  have used that
\begin{eqnarray*}
\int_{-a}^{a}x^{2n+1}G^{\prime}(x)\,dx=0,
\end{eqnarray*}
because $G^{\prime}(x)$ is an even function and $x^{2n+1}$ is odd.

Thus, if $a,b>0$, we have,
\begin{equation}\label{cat}
\int_{-a}^{a}\int_{-b}^{b}(x_1 x_2)^{2n}G_1(x_1)G_2(x_2)\,dx_1\,dx_2=\frac{(a b)^{2n+1}}{(2n+1)^2}.
\end{equation}

Now, by multiplying on both sides of (\ref{cat}) by $(-1)^n$, $n=0,1,\dots$ and summing, straightforward computations provide that
\begin{eqnarray*}
\sum_{n=0}^{\infty}\frac{(-1)^n(a b)^{2n+1}}{(2n+1)^2} &=& \int_{-a}^{a}\int_{-b}^{b}\left(\sum_{n=0}^{\infty}(-1)^n(x_1 x_2)^{2n}\right)G_1(x_1)G_2(x_2)\,dx_1\,dx_2\\
&=& \int_{-a}^{a}\int_{-b}^{b}\frac{G_1(x_1)G_2(x_2)}{1+x_1^2x_2^2}\,dx_1\,dx_2.
\end{eqnarray*}

Taking now $b=1/a$ we get
\begin{eqnarray*}
\int_{-a}^{a}\int_{-1/a}^{1/a}\frac{G_1(x_1)G_2(x_2)}{1+x_1^2x_2^2}\,dx_1\,dx_2=\sum_{n=0}^{\infty}\frac{(-1)^n}{(2n+1)^2}={\bf G},
\end{eqnarray*}
where ${\bf G}$ is the Catalan's constant.
\end{proof}

\noindent \textbf{Remark 1} Note that the functions defined in Theorem \ref{th1} correspond to cumulative distribution functions (cdfs), which are widely used in probability theory. The additional condition in Theorem \ref{th1} allows us to conclude that the derivative of these functions is an even function.\\

\noindent \textbf{Remark 2} Additionally, it is merely to note that the value of $a$ in (\ref{eq1}) can be replaced for any positive real number, such as $\pi$, $e$, including the proper Catalan's constant. For example we have that
\begin{eqnarray*}
{\bf G}=\int_{-{\bf G}}^{{\bf G}}\int_{-1/{\bf G}}^{1/{\bf G}}\frac{G_1(x_1)G_2(x_2)}{1+x_1^2x_2^2}\,dx_1\,dx_2.
\end{eqnarray*}

\subsection{Double integral examples}

\noindent The following examples illustrate the relevance of Theorem \ref{th1}. We use Theorem \ref{th1} taking $a=1$ in all following cases.\\

\noindent{\bf Example 1: A first double integral.} Entry (40) in \cite{Bradley2001} establishes,
\begin{equation}\label{entry40}
{\bf G}=\int_{0}^{1}\int_{0}^{1}\frac{dxdy}{1+x^2y^2}.
\end{equation}

This formula can be proved directly from (\ref{eq1}) by considering the function (Rademacher distribution) defined as $G_1(x)=G_2(x)=\frac{1}{2}$ if $-1\le x< 1$, 0 if $x< -1$ and 1 if $x\ge 1$. Since $G_i(x)$, $i=1,2$ satisfy the conditions of Theorem \ref{th1}, we have
\begin{eqnarray*}
{\bf G}=\int_{-1}^{-1}\int_{-1}^{1}\frac{1/4}{1+x^2y^2}dxdy &=&  4\int_{0}^{1}\int_{0}^{1}\frac{1/4}{1+x^2y^2}dxdy\\
   &=& \int_{0}^{1}\frac{\tan^{-1}(y)}{y}dy={\bf G}.
\end{eqnarray*}
Note that (\ref{entry40}) can also be evaluated using Theorem 1 in Glasser (see \cite{Glasser2019}),
$$
\int_{0}^{1}\int_{0}^{1}\frac{dxdy}{1+x^2y^2}=-\int_{0}^{1}\frac{\ln x}{1+x^2}dx={\bf G},
$$
where this last integral corresponds to the entry (16) in \cite{Bradley2001}.\\

\noindent{\bf Example 2.} From (\ref{entry40}) and by making the change of variable $x=i z$, $y=i w$, where $i=\sqrt{-1}$, it is easy to get
\begin{eqnarray*}
{\bf G}=-\int_{0}^{1/i}\int_{0}^{1/i}\frac{dz\,dw}{1+z^2 w^2}.
\end{eqnarray*}\\

\noindent{\bf Example 3: A new double integral.} We have next nice double integral,
\begin{equation}\label{eqhs}
{\bf G}=\int_{-1}^{1}\int_{-1}^{1}\frac{4\tan^{-1}e^x\tan^{-1}e^y}{\pi^2(1+x^2y^2)}\,dx\,dy.
\end{equation}

Formula (\ref{eqhs}) is obtained using in Theorem \ref{th1} the functions $G_i(x)=\frac{2}{\pi}\tan^{-1}(e^x)$, with $x\in \mathbb{R}$, $i=1,2$, which correspond with the cdf of the hyperbolic secant distribution, according to \cite{Holst2013}.\\

\noindent{\bf Example 4: Catalan's constant and the $\Phi$ function.} If we consider the $\Phi(x)$ function,
$$
\Phi(x)=\frac{1}{\sqrt{2\pi}}\int_{-\infty}^{x}\exp(-t^2/2)\,dt,
$$
with $x\in\mathbb{R}$, corresponding to the cumulative distribution function of the standard normal distribution, using (\ref{th1}), we have
$$
{\bf G}=\int_{-1}^{1}\int_{-1}^{1}\frac{\Phi(x)\Phi(y)}{1+x^2y^2}\,dx\,dy.
$$

\noindent{\bf Example 5: Another double integral.} Applying Theorem \ref{th1} to the functions $G_1(x)=\frac{1}{2}+\frac{1}{\pi}\tan^{-1}x$ and $G_2(y)=\frac{1}{2}$, we obtain a new representation of the Catalan's constant,
\begin{eqnarray*}\label{entry46}
{\bf G}=\int_{-1}^{1}\int_{-1}^{1}\frac{\frac{1}{2}\left(\frac{1}{2}+\frac{1}{\pi}\tan^{-1}x\right)}{1+x^2y^2}\,dx\,dy.
\end{eqnarray*}

Now, we can write
$$
{\bf G}=\int_{-1}^{1}\int_{-1}^{1}\frac{1/4}{1+x^2y^2}\,dx\,dy+\int_{-1}^{1}\int_{-1}^{1}\frac{(1/2\pi)\tan^{-1}x}{1+x^2y^2}\,dx\,dy.
$$

Now, since the first of the integrals is the Catalan's constant, the value of the second of the integrals is zero,
$$\int_{-1}^{1}\int_{-1}^{1}\frac{\tan^{-1}x}{1+x^2y^2}\,dx\,dy=0.$$

The above integral can be written as the sum of four double integrals corresponding to the four quadrants. Using entry (46) in \cite{Bradley2001} we have,
$$
\int_{0}^{1}\int_{0}^{1}\frac{\tan^{-1}x}{1+x^2y^2}\,dx\,dy=\int_{0}^{1}\int_{-1}^{0}\frac{\tan^{-1}x}{1+x^2y^2}\,dx\,dy=
\frac{\pi{\bf G}}{2}-\frac{7}{8}\zeta(3),$$
where $\zeta(3)$ is the Ap\'ery constant. For the other two integrals, its value is the same and is equal to
$$
\int_{-1}^{0}\left(\int_{0}^{1}\frac{\tan^{-1}x}{1+x^2y^2}\,dy\right)\,dx=\int_{-1}^{0}\frac{(\tan^{-1}x)^2}{x}\,dx=-\frac{\pi{\bf G}}{2}+\frac{7}{8}\zeta(3).
$$

\noindent{\bf Example 6.} Let $G(x)=(\alpha^3+x^3)/(2\alpha^3)$, $\alpha\in\mathbb{R}-\{0\}$, which corresponds to the cdf of the U-quadratic distribution with domain in $(-\alpha,\alpha)$. Then, we have the following representation of the Catalan's constant,
\begin{eqnarray*}
{\bf G} = \frac{1}{4\alpha^6}\int_{-1}^{1}\int_{-1}^{1}\frac{(\alpha^3+x^3)(\alpha^3+y^3)}{1+x^2 y^2}\,dx\,dy,\quad \alpha\in\mathbb{R}-\{0\}.
\end{eqnarray*}

\section{A single integral representation of the Catalan's constant}\label{section3}

The following Theorem provides a new general representation of the Catalan constant in terms of a simple integral making use of the distribution functions, defined in Theorem \ref{th1}.

\begin{theorem}
We have,
\begin{equation}\label{eq2}
{\bf G}=\int_{-1}^{1}\frac{G(x)\tan^{-1}x}{x}\,dx,
\end{equation}
where $G(x)$ is a function satisfying the conditions of the Theorem \ref{th1}.
\end{theorem}
\begin{proof}
To prove (\ref{eq2}) we consider the function $G_1(y)=\frac{1}{2}\left(1+y+\frac{1}{\pi}\sin(\pi y)\right)$, with $-1\le y\le 1$ and $G_2(x)=G(x)$. Then,
$$
\int_{-1}^{1}\frac{G_1(y)G(x)}{1+x^2y^2}dy=\frac{G(x)}{2}\int_{-1}^{1}\left(\frac{1}{1+x^2y^2}+\frac{y+\frac{1}{\pi}\sin(\pi y)}{1+x^2y^2}\right)dy=G(x)\frac{\tan^{-1}x}{x},
$$
and integrating in the variable $x$ and using Theorem \ref{th1}, we obtain (\ref{eq2}).
\end{proof}\\

Using expression (\ref{eq2}), we can obtain the following new representations of the Catalan's constant,
\begin{eqnarray}
  {\bf G} &=& \int_{-1}^{1}\frac{(1+x)\tan^{-1}x}{2x}\,dx, \label{IntegraG1}\\
  {\bf G} &=& \int_{-1}^{1}\frac{\left(\frac{1}{2}+\frac{1}{\pi}\tan^{-1}x\right)\tan^{-1}x}{x}\,dx,\label{IntegraG2} \\
  {\bf G} &=& \int_{-1}^{1}\frac{\left(\frac{\pi}{2}+\sin^{-1}x\right)\tan^{-1}(x)}{\pi x}\,dx, \label{IntegraG3}\\
  {\bf G} &=& \int_{-1}^{1}\frac{\left(1+{\mbox{erf}}\left(\frac{x}{\sqrt{2}}\right) \right)\tan^{-1}(x)}{2x}\,dx.\label{IntegraG4}
\end{eqnarray}\\
The chosen $G$ functions have been, $G(x)=\frac{x+1}{2}$, $-1\le x\le 1$ for (\ref{IntegraG1}); $G(x)=\frac{1}{2}+\frac{1}{\pi}\tan^{-1}x$ for (\ref{IntegraG2}), corresponding to the cdf of the Cauchy distribution; $G(x)=\frac{1}{\pi}\left(\frac{\pi}{2}+\sin^{-1}x\right)$ in (\ref{IntegraG3}), corresponding to the arcsin distributions and $G(x)=\frac{1}{2}\left(1+{\mbox{erf}}\left(\frac{x}{\sqrt{2}}\right)\right)$, where $\mbox{erf}(x)$ is the error function, for obtaining (\ref{IntegraG4}).

\section{General representation of the Catalan constant by means of multiple integrals}\label{section4}

A more general result is given now than the one provided in Theorem \ref{th1}.
\begin{theorem}\label{th2}
Let $G_i:\mathbb{R}\to [0,1]$, $i=1,2,\dots,r$ be $r$ right-continuous monotone non-decreasing functions satisfying $\lim_{x\to -\infty}G_i(x)=0$ and $\lim_{x\to\infty}G_i(x)=1$, $r=1,2,\dots$ If $G_i(x_i)=1-G_i(-x_i)$ and $a_i>0$, $i=1,2,\dots,r$, with $a_r=\left\{\prod_{i=1}^{r-1}a_i\right\}^{-1}$ and being $\prod_{i=1}^{0}a_i=1$, we have,
\begin{equation}\label{eq3}
{\bf G}=2^{r-2}\int_{-a_1}^{a_1}\int_{-a_2}^{a_2}\cdots\int_{-a_r}^{a_r}\Phi\left(-\prod_{i=1}^{r}x_i^2,2-r,\frac{1}{2}\right)
\prod_{i=1}^{r}G_i(x_i)\,dx_i,
\end{equation}
where
\begin{equation}\label{lerch}
\Phi(z,s,a)=\sum_{k=0}^{\infty}\frac{z^k}{(k+a)^{s}},
\end{equation}
is the Lerch transcendent function, with $a>0$, with $|z|<1$ or ${\cal R}(s)>1$ and $|z|=1$, and ${\bf G}$ is the Catalan's constant.
\end{theorem}
\begin{proof} We have that
\begin{equation}\label{cat1}
\int_{-a_1}^{a_1}\cdots\int_{-a_r}^{a_r}\left(\prod_{i=1}^{r}x_i\right)^{2n}\prod_{i=1}^{r}G_i(x_i)dx_i=\frac{\left(\prod_{i=1}^{r}a_i\right)^{2n+1}}{(2n+1)^r}.
\end{equation}

Now, by multiplying on both sides of (\ref{cat1}) by $(-1)^n (2n+1)^{r-2}$, $n=0,1,\dots$ and summing, straightforward computations provide
\begin{eqnarray*}
\sum_{n=0}^{\infty}\frac{(-1)^n\left(\prod_{i=1}^{r}a_i\right)^{2n+1}}{(2n+1)^2} &=& \int_{-a_1}^{a_1}\cdots\int_{-a_r}^{a_r}\left[\sum_{n=0}^{\infty}\frac{(-1)^n } {(2n+1)^{2-r}}\left(\prod_{i=1}^{r}x_i\right)^{2n}\right]\prod_{i=1}^{r}G_i(x_i)\,dx_i\\
&=& \int_{-a_1}^{a_1}\cdots\int_{-a_r}^{a_r}\left[2^{r-2}\sum_{n=0}^{\infty}\left(n+\frac{1}{2}\right)^{r-2}\left(-\prod_{i=1}^{r}x_i^2\right)^{n}\right]
\prod_{i=1}^{r}G_i(x_i)\,dx_i\\
&=& 2^{r-2}\int_{-a_1}^{a_1}\cdots\int_{-a_r}^{a_r}\Phi\left(-\prod_{i=1}^{r}x_i^2,2-r,\frac{1}{2}\right)\prod_{i=1}^{r}G_i(x_i)\,dx_i.
\end{eqnarray*}
Finally, as $a_r=\left\{\prod_{i=1}^{r-1}a_i\right\}^{-1}$ and using (\ref{lerch}) we obtain the result.
\end{proof}

\noindent{\bf Remark:} The case $r=1$ gives expression (\ref{eq2}). For the special case $r=2$ expression (\ref{eq3}) reduces to expression (\ref{eq1}). \\

\noindent{\bf Example} In the case of $r=3$, formula (\ref{eq3}) needs to compute the Lerch's formula given by,
$$\Phi\left(-z,-1,\frac{1}{2}\right)=\frac{1-z}{2(1+z)^2},$$
and we obtain the following general representation of the Catalan constant by means of the triple integral,

\begin{eqnarray*}
{\bf G}=\int_{-a_1}^{a_1}\int_{-a_2}^{a_2}\int_{-1/(a_1 a_2)}^{1/(a_1 a_2)} \frac{(1-x_1^2x_2^2x_3^2)G_1(x_1)G_2(x_2)G_3(x_3)}{(1+x_1^2x_2^2x_3^2)^2}\,dx_1\,dx_2\,dx_3,
\end{eqnarray*}\\
where $G_i(z)$ are defined in Theorem \ref{th1} and \ref{th2}. In the Appendix, we include the expressions of the Lerch function given in (\ref{eq3}), for $r=3,4\dots,10$

\subsection{Catalan multiple integral representations in dimensions higher than two}

In this section, we will obtain some simple and relevant representations of the Catalan's constant in dimensions greater than two, where the integrand corresponds to polynomial rational functions. We will use Theorem \ref{th2}, taking again the Rademacher distribution function, defined as $G_i(x_i)=\frac{1}{2}$ if $-1\le x_i< 1$, 0 if $x_i< -1$
and 1 if $x_i\ge 1$, $i=1,\dots,r$ and $a_i=1$. We use the notation $\bm x=(x_1,\dots,x_r)$ and $d\bm x=(dx_1,\dots,dx_r)$ and we represent,
\begin{equation}\label{zi}
x_{(i)}=\prod_{j=1}^{i}x_j^2,\;i=1,2,\dots,r.
\end{equation}

Then, in dimension $r=3$
we have the formula,
\begin{eqnarray*}\label{dim3}
{\bf G}=\frac{1}{2^3}\int_{[-1,1]^3}\frac{1-x_{(3)}}{(1+x_{(3)})^2}d\bm x=
\int_{[0,1]^3}\frac{1-x_{(3)}}{(1+x_{(3)})^2}\,d\bm x,
\end{eqnarray*}
where $x_{(3)}=(x_1x_2x_3)^2$ according to (\ref{zi}). The second integral with limits in the interval $[0,1]$
is obtained by symmetry. Now, for dimension $r=4$
\begin{eqnarray*}\label{dim4}
{\bf G}=\frac{1}{2^4}\int_{[-1,1]^4}\frac{1-6x_{(4)}+x^2_{(4)}}{(1+x_{(4)})^3}\,d\bm x=
\int_{[0,1]^4}\frac{1-6x_{(4)}+x^2_{(4)}}{(1+x_{(4)})^3}\,d\bm x,
\end{eqnarray*}
where $x_{(4)}$ is defined in (\ref{zi}).

The formula for dimension $r=5$ is,
\begin{eqnarray*}\label{dim5}
{\bf G}=\frac{1}{2^5}\int_{[-1,1]^5}\frac{1-23x_{(5)}+23x_{(5)}-x^2_{(5)}}{(1+x_{(5)})^4}d\bm x=\int_{[0,1]^5}\frac{1-23x_{(5)}+23x_{(5)}-x^2_{(5)}}{(1+x_{(5)})^4}\,d\bm x,
\end{eqnarray*}
and in dimension $r=6$
\begin{eqnarray*}
{\bf G} &=& \frac{1}{2^6}\int_{[-1,1]^6}\frac{1-76x_{(6)}+230x^2_{(6)}-76x^3_{(6)}+x^4_{(6)}}{(1+x_{(6)})^5}\,d\bm x \label{dim6}\\
   &=& \int_{[0,1]^6}\frac{1-76x_{(6)}+230x^2_{(6)}-76x^3_{(6)}+x^4_{(6)}}{(1+x_{(6)})^5}\,d\bm x,
\end{eqnarray*}
where $x_{(5)}$ and $x_{(6)}$ are defined in (\ref{zi}).\\

Finally, the representation of the Catalan constant by means of an integral of dimension $r=10$ is,
\begin{eqnarray*}
{\bf G} = \frac{1}{2^{10}}\int_{[-1,1]^{10}}\frac{1+\sum_{j=1}^{8}a_jx_{(10)}^j}{(1+x_{(10)})^5}\,d\bm x = \int_{[0,1]^{10}}\frac{1+\sum_{j=1}^{8}a_jx_{(10)}^j}{(1+x_{(10)})^5}\,d\bm x
\end{eqnarray*}
where,
$a_1=-6552$, $a_2=331612$, $a_3=-2485288$, $a_4=4675014$, and $a_r=a_{9-r}$, $r=1,\dots,8$, and $x_{)10)}$ is again defined in (\ref{zi}).

\section*{Conflict of Interest}

The authors of this paper declare that they have no conflicts of interest.

\section*{Appendix: expressions of the Lerch function in Theorem \ref{th2}}

In this Appendix we include some expressions of the Lerch function in Theorem \ref{th2}, for values $\Phi(-z,j,\frac{1}{2})$ with $j=-1,\dots,-8$, which correspond to multiple integrals in dimensions $\mathbb{R}^n$, with $n=3,4,\dots,10$.\\

\noindent $r=3$
$$\Phi\left(-z,-1,\frac{1}{2}\right)=\frac{1-z}{2(1+z)^2}.$$

\noindent $r=4$
$$\Phi\left(-z,-2,\frac{1}{2}\right)=\frac{1-6z+z^2}{4(1+z)^3}.$$

\noindent $r=5$
$$\Phi\left(-z,-3,\frac{1}{2}\right)=\frac{1-23z+23z^2-z^3}{8(1+z)^4}.$$

\noindent $r=6$
$$\Phi\left(-z,-4,\frac{1}{2}\right)=\frac{1-76z+230z^2-76z^3+z^4}{16(1+z)^5}.$$

\noindent $r=7$
$$\Phi\left(-z,-5,\frac{1}{2}\right)=\frac{-z^5+237 z^4-1682 z^3+1682 z^2-237 z+1}{32 (z+1)^6}.$$

\noindent $r=8$
$$\Phi\left(-z,-6,\frac{1}{2}\right)=\frac{z^6-722 z^5+10543 z^4-23548 z^3+10543 z^2-722 z+1}{64 (z+1)^7}.$$

\noindent $r=9$
$$\Phi\left(-z,-7,\frac{1}{2}\right)=\frac{-z^7+2179 z^6-60657 z^5+259723 z^4-259723 z^3+60657 z^2-2179 z+1}{128 (z+1)^8}.$$

\noindent $r=10${\small
$$\Phi\left(-z,-8,\frac{1}{2}\right)=\frac{(1 - 6552 z + 331612 z^2 - 2485288 z^3 + 4675014 z^4 - 2485288 z^5 +
 331612 z^6 - 6552 z^7 + z^8)}{256 (1 + z)^9}.$$}

\end{document}